\def\draftdate{January 30, 2020}

\documentclass{amsart}
\usepackage{amssymb}
\usepackage{xy}
\usepackage{hyperref}

\newcommand{\indexing}{\nu\in S}

\newcommand{\supdot}{^{\bullet}}

\newcommand{\HS}[1]{\textrm{Hyp}(#1, \Sp)}
\newcommand{\PS}[1]{\textrm{Pre}(#1, \Sp)}
\newcommand{\Site}{\oT}
\newcommand{\Ab}{\aA b}
\newcommand{\AbSh}[1]{\textrm{Sh}(#1,\Ab)}

\newcommand{\HFC}{\Fib(\kappa)}

\newcommand{\NR}{\aO}
\newcommand{\RS}{\NR_{F}[\tfrac1p]}
\newcommand{\et}{\textup{\'et}}
\newcommand{\LKF}{\aK}
\newcommand{\LKFn}[1][{n}]{\aK^{/p^{#1}}}
\newcommand{\LKM}[1]{\aK_{#1}}
\newcommand{\LKMn}[2]{\aK_{#1}^{/p^{#2}}}

\newcommand{\LK}{\LKM{U_{\et}}}
\newcommand{\LKn}[1][{n}]{\LKMn{U_{\et}}{#1}}

\newcommand{\SHom}{\mathop{\aH\mathrm{om}}}

\let\iso\cong
\let\sma\wedge
\let\smaL\sma

\renewcommand{\to}{\mathchoice{\longrightarrow}{\rightarrow}{\rightarrow}{\rightarrow}}

\newcommand{\sto}{\rightarrow}
\newcommand{\overto}[1]{\xrightarrow{\,#1\,}}

\newcommand{\phat}{^{\scriptscriptstyle\wedge}_{p}}

\newcommand{\nucomp}{^{\wedge}_{\nu}}

\let\catsymbfont\mathcal
\newcommand{\aA}{{\catsymbfont{A}}}

\newcommand{\aF}{{\catsymbfont{F}}}

\newcommand{\aG}{{\catsymbfont{G}}}
\newcommand{\aH}{{\catsymbfont{H}}}
\newcommand{\aK}{{\catsymbfont{K}}}

\newcommand{\aO}{{\catsymbfont{O}}}

\newcommand{\aX}{{\catsymbfont{X}}}

\newcommand{\bG}{{\mathbb{G}}}

\newcommand{\bH}{{\mathbb{H}}}

\newcommand{\bS}{{\mathbb{S}}}
\newcommand{\bZ}{{\mathbb{Z}}}

\newcommand{\bZp}{{\mathbb{Z}_{p}}}
\newcommand{\bZpi}{\bQp/\bZp}

\newcommand{\bQ}{{\mathbb{Q}}}

\newcommand{\bQp}{{\mathbb{Q}_{p}}}

\newcommand{\oT}{\mathcal{T}}

\newcommand{\SL}{\textrm{SL}}

\def\quickop#1{\expandafter\DeclareMathOperator\csname
#1\endcsname{#1}}
\quickop{id}\quickop{Id}
\quickop{Tor}\quickop{Ext}
\quickop{holim}\quickop{hocolim}\quickop{hofib}
\quickop{Tel}
\quickop{Pic}\quickop{Gal}
\quickop{trc}
\quickop{pre}
\quickop{spec}
\quickop{colim}
\quickop{sd}
\quickop{dom}
\quickop{cod}
\quickop{Hom}
\quickop{Sp}
\quickop{op}
\quickop{sk}
\quickop{cosk}
\quickop{disc}
\quickop{Tot}
\quickop{Mod}
\quickop{Fib}
\quickop{Cof}
\quickop{Diff}
\quickop{ab}
\quickop{nr}
\quickop{coker}
\quickop{rec}

\numberwithin{equation}{section}

\newtheorem{thm}[equation]{Theorem}
\newtheorem*{main}{Theorem}
\newtheorem*{maincor}{Corollary}

\theoremstyle{definition}

\theoremstyle{remark}

\xyoption{arrow}
\xyoption{matrix}
\xyoption{cmtip}
\SelectTips{cm}{}

\newcommand{\term}[1]{\textit{#1}}

\newdir{ >}{{}*!/-5pt/\dir{>}}

\begin{document}

\title[$K$-theoretic Tate-Poitou duality]
{$K$-theoretic Tate-Poitou duality and the fiber of the cyclotomic trace}

\author{Andrew J. Blumberg}
\address{Department of Mathematics, The University of Texas,
Austin, TX \ 78712}
\email{blumberg@math.utexas.edu}
\thanks{The first author was supported in part by NSF grants
DMS-1151577, DMS-1812064}
\author{Michael A. Mandell}
\address{Department of Mathematics, Indiana University,
Bloomington, IN \ 47405}
\thanks{The second author was supported in part by NSF grants
DMS-1505579, DMS-1811820}
\email{mmandell@indiana.edu}

\date{\draftdate} 

\begin{abstract}
Let $p\in \bZ$ be an odd prime.  We prove a spectral version of
Tate-Poitou duality for the algebraic $K$-theory spectra of number
rings with $p$ inverted.  This identifies the
homotopy type of the fiber of the cyclotomic trace $K(\NR_{F})\phat
\to TC(\NR_{F})\phat$ after taking a suitably connective cover.  As an
application, we identify the homotopy 
type at odd primes of the homotopy fiber of the cyclotomic trace for the sphere
spectrum in terms of the algebraic $K$-theory of $\bZ$.
\end{abstract}

\maketitle

\section*{Introduction}

Tate-Poitou duality describes the relationship between the \'etale
cohomology of $S$-integers in number fields  and their
completions in terms of a long exact sequence where the third term is
a Pontryagin dual related to the first term.  In the most basic case,
for $p>2$ a prime in $\bZ$ and a number field $F$, we get a long exact
sequence 
\begin{equation*}\label{eq:tatepoitou}
\begin{gathered}
\xymatrix@R-1pc@C-1.25pc{%
0\ar[r]
&H^{0}_{\et}(\NR_{F}[\tfrac1p];\bZ\phat(k))\ar[r]
&\prod\limits_{\nu\mid p}H^{0}_{\et}(F\nucomp;\bZ\phat(k))\ar[r]
&(H^{2}_{\et}(\NR_{F}[\tfrac1p],\bZ/p^{\infty}(1-k)))^{*}
\ar `r/.75pc[d] `d[l] `[llld] `d/.5pc[lld] [lld]
\\
&H^{1}_{\et}(\NR_{F}[\tfrac1p];\bZ\phat(k))\ar[r]
&\prod\limits_{\nu\mid p}H^{1}_{\et}(F\nucomp;\bZ\phat(k))\ar[r]
&(H^{1}_{\et}(\NR_{F}[\tfrac1p],\bZ/p^{\infty}(1-k)))^{*}
\ar `r/.75pc[d] `d[l] `[llld] `d/.5pc[lld] [lld]
\\
&H^{2}_{\et}(\NR_{F}[\tfrac1p];\bZ\phat(k))\ar[r]
&\prod\limits_{\nu\mid p}H^{2}_{\et}(F\nucomp;\bZ\phat(k))\ar[r]
&(H^{0}_{\et}(\NR_{F}[\tfrac1p],\bZ/p^{\infty}(1-k)))^{*}\ar[r]&0
}
\end{gathered}
\end{equation*}
where $\NR_{F}$ denotes the ring of integers, $F\nucomp$
denotes completion at the valuation $\nu$, and $(-)^{*}$ denotes
Pontryagin dual.  The purpose of this paper is to describe a spectrum-level
``$K$-theoretic'' version of Tate-Poitou duality encoding the behavior
of the completion map in the algebraic $K$-theory of rings of integers
in number fields and use it to study the algebraic $K$-theory of the
sphere spectrum.

Thomason's work \cite{ThomasonEtale} on the Quillen-Lichtenbaum
conjecture identifies the \'etale cohomology groups in the above
sequence as the homotopy groups of
the \term{$K(1)$-localization} of algebraic $K$-theory spectra.
Specifically, Thomason~\cite[Theorem~4.1, App.~A]{ThomasonEtale} shows
\[
\pi_{n}(L_{K(1)}K(R))\iso
\begin{cases}
H^{0}_{\et}(R;\bZp(\tfrac n2))\oplus
H^{2}_{\et}(R;\bZp(\tfrac n2+1))
&n\text{ even}\\[2pt]
H^{1}_{\et}(R;\bZp(\tfrac {n+1}2))
&n\text{ odd}\\
\end{cases}
\]
for $R=\NR_{F}[1/p]$ or $F\nucomp$.  Letting $M_{\bZ/p^{\infty}}$
denote the Moore spectrum for $\bZ/p^{\infty}$, we also have
\[
\pi_{n}(L_{K(1)}K(R)\sma M_{\bZ/p^{\infty}})\iso
\begin{cases}
H^{0}_{\et}(R;\bZ/p^{\infty}(\tfrac n2))\oplus
H^{2}_{\et}(R;\bZ/p^{\infty}(\tfrac n2+1))
&n\text{ even}\\[2pt]
H^{1}_{\et}(R;\bZ/p^{\infty}(\tfrac {n+1}2))
&n\text{ odd}\\
\end{cases}
\]
for $R=\NR_{F}[1/p]$. 
Algebraically, we can then
use these isomorphisms to rewrite the Tate-Poitou sequence as the long
exact sequence
\[
\xymatrix@R-1pc@C-1.25pc{%
&&
&\hspace{-10em}\cdots \to(\pi_{-1-n}(L_{K(1)}K(\NR_{F}[1/p])\sma M_{\bZ/p^{\infty}}))^{*}
\ar `r/.5pc[d] `d[l] `[llld] `d/.5pc[lld] [lld]
\\
&\pi_{n}L_{K(1)}K(\NR_{F}[1/p])\ar[r]
&\prod\pi_{n}L_{K(1)}K(F\nucomp)\to\cdots &
}
\]
with the first term (on the left) the homotopy groups of the $K(1)$-localization of
$K(\NR_{F}[1/p])$ and the second term the homotopy groups of the
product of the $K(1)$-localizations of the $K$-theory of the completed
fields.  We can interpret the third term as the homotopy groups of a
spectrum as well, using Brown-Comenetz duality or Anderson duality:
\begin{align*}
\pi_{-1-n}(L_{K(1)}K(\NR_{F}[1/p])\sma M_{\bZ/p^{\infty}}))^{*}
&\iso \pi_{n}(\Sigma^{-1}I_{\bQ/\bZ}(L_{K(1)}K(\bZ)\sma M_{\bZ/p^{\infty}}))\\
&\iso \pi_{n}(\Sigma^{-1}I_{\bZ\phat} L_{K(1)}K(\bZ)),
\end{align*}
where $I_{\bQ/\bZ}$ denotes the Brown-Comenetz dual and $I_{\bZ\phat}$
denotes the Anderson dual of $p$-complete spectra.  Our main result
lifts this exact sequence to a cofiber sequence on the spectrum
level.

\begin{main}[$K$-Theoretic Tate-Poitou Duality]\label{thm:main}
Let $p>2$ be a prime number. 
Let $F$ be a number field, $\NR_{F}$ its ring of integers, and $S$ the
set of primes of $\NR_{F}$ above $p$.  For $\nu
\in S$, write $F\nucomp$ for the $\nu$-completion of $F$.
The homotopy fiber $\HFC$ of the completion map in $K(1)$-local algebraic $K$-theory 
\[
\kappa \colon 
L_{K(1)}K(\RS) \to 
\prod_{\indexing} 
L_{K(1)}K(F\nucomp)
\]
is weakly equivalent to 
\[
\Sigma^{-1}I_{\bQ/\bZ}(L_{K(1)}K(\RS)\sma M_{\bZpi})
\simeq \Sigma^{-1}I_{\bZp} L_{K(1)}K(\RS).
\]
The weak equivalence $\HFC\to \Sigma^{-1}I_{\bZp}L_{K(1)}K(\RS)$ is adjoint to the map 
\[
L_{K(1)}K(\RS)\sma \HFC\to \Sigma^{-1} I_{\bZp}\bS
\]
induced by the
$L_{K(1)}K(\RS)$-module structure map $L_{K(1)}K(\RS)\sma \HFC\to \HFC$
and a map 
\[
u_{\NR_{F}}\colon \HFC\to \Sigma^{-1} I_{\bZp}\bS
\]
constructed as~\eqref{eq:u}.
\end{main}

The map $u_{\NR_{F}}$ above is ``canonical'' in that given the standard
conventions for the Hasse invariant
(see~\cite[XIII\S3,p.~193]{Serre-LocalFields}), the construction
involves no further choices.

The previous theorem establishes a global arithmetic duality for algebraic
$K$-theory.  Clausen's MIT thesis~\cite{Clausen-Thesis} contains work
in this direction; in particular, Clausen produces a duality map of a
similar type to the one in the main theorem (with different details),
presumably related through Gross-Hopkins duality.  Work in progress of
Schlank and Stojanoska~\cite{SchlankStojanoska} establishes arithmetic
duality results for a much wider range of theories.

For us, the main interest is in the case $F=\bQ$, 
$S=\{p\}$,
where the main theorem (and some fiddling in low dimensions)
identifies the homotopy fiber of the cyclotomic trace on the sphere
spectrum. 

\begin{maincor}
Let $p$ be an odd prime. The connective cover of the homotopy fiber of the
cyclotomic trace $K(\bS)\phat\to TC(\bS)\phat$ is weakly equivalent to
the connective cover of $\Sigma^{-1}I_{\bZ\phat}(L_{K(1)}K(\bZ))$.
\end{maincor}

Rognes~\cite{Rognes2,Rognesp} had previously identified the homotopy
type of the homotopy fiber of the cyclotomic trace at regular primes,
although not in these terms.

To deduce the corollary from the main theorem, we apply the work of
Dundas~\cite{Dundas-RelK} and Hesselholt-Madsen~\cite{HM2} together with
the (affirmed) Quillen-Lichtenbaum conjecture.  The linearization map
$\bS\to \bZ$ and the cyclotomic trace $K \to TC$ fit together in a
commutative square
\[
\xymatrix{%
K(\bS)\phat\ar[r]\ar[d]&K(\bZ)\phat\ar[d]\\
TC(\bS)\phat\ar[r]&TC(\bZ)\phat
}
\]
that Dundas~\cite{Dundas-RelK} shows is homotopy cartesian, and it
follows that the homotopy fiber of the cyclotomic trace for $\bS$ is
weakly equivalent to the homotopy fiber of the cyclotomic trace for $\bZ$.
Hesselholt-Madsen~\cite{HM2} shows that the completion map $TC(\bZ)\phat\to
TC(\bZ\phat)\phat$ is a weak equivalence and the cyclotomic trace
$K(\bZ\phat)\phat\to TC(\bZ\phat)\phat$ is a connective cover.  It
follows that the connective cover of the homotopy fiber of the
cyclotomic trace is weakly equivalent to the homotopy fiber of the
completion map $K(\bZ)\phat\to K(\bZ\phat)\phat$.  Quillen's
localization sequence for $\bZ\to \bZ[\tfrac1p]$ and $\bZ\phat\to \bQ\phat$,
\[
\xymatrix{%
K(\bZ/p)\ar[r]\ar[d]_{\id}&K(\bZ)\ar[r]\ar[d]&K(\bZ[1/p])\ar[r]\ar[d]&\Sigma\dotsb \\
K(\bZ/p)\ar[r]&K(\bZ\phat)\ar[r]&K(\bQ\phat)\ar[r]&\Sigma\dotsb
}
\]
then shows that the homotopy fiber of the completion map
$K(\bZ)\to K(\bZ\phat)$ is weakly equivalent to the homotopy fiber of
the completion map $K(\bZ[\tfrac1p])\phat\to K(\bQ\phat)\phat$.
The (affirmed) Quillen-Lichtenbaum conjecture~\cite[VI.8.2]{Weibel-KBook}
implies that the maps  
\[
K(\bZ[\tfrac1p])\phat\to L_{K(1)}K(\bZ[\tfrac1p]), \qquad
K(\bQ\phat)\phat\to L_{K(1)}K(\bQ\phat), \qquad
\]
are weak equivalences after taking $1$-connected covers.  Looking in
low dimensions, $K(\bZ[\tfrac1p])\phat\to L_{K(1)}K(\bZ[\tfrac1p])$ is
actually a connective cover,
while $K(\bQ\phat)\phat\to
L_{K(1)}K(\bQ\phat)$ induces an isomorphism on homotopy groups in
degrees $>0$ and an injection in degree $0$.  It then follows that the
(connective) homotopy fiber of $K(\bZ[\tfrac1p])\phat\to
K(\bQ\phat)\phat$ is weakly equivalent to the connective cover of the
homotopy fiber of $L_{K(1)}K(\bZ[\tfrac1p])\to L_{K(1)}K(\bQ\phat)$.
Applying the main theorem, the corollary now follows.

Returning to the case of a number field $F$,
the part of the discussion above that is general still
obtains. Writing $R\nucomp$ for the completion of $\NR_{F}$ at
the prime $\nu$, we have Quillen's localization sequence for $\NR_{F}\to
\RS$ and $\prod R\nucomp\to \prod F\nucomp$,
\[
\xymatrix@R-1pc{%
\prod\limits_{\vtop to 0pt{\hbox{$\scriptstyle\indexing$}
\hrule height0pt width0pt depth0pt
\vss}} 
 K(\NR_{F}/\nu)\ar[r]\ar[d]^-{\iso}&K(\NR_{F})\ar[r]\ar[d]&K(\RS)\ar[r]\ar[d]&\Sigma\dotsb \\
\prod\limits_{\indexing} 
K(R\nucomp/\nu)\ar[r]&
K(
\prod\limits_{\indexing} 
R\nucomp)\ar[r]&
K(
\prod\limits_{\indexing} 
F\nucomp)\ar[r]&\Sigma\dotsb,
}
\]
which identifies the homotopy fiber of the completion map $K(\NR_{F})\to
K(\prod R\nucomp)$ as weakly equivalent to the homotopy fiber
of the completion map $K(\RS)\to K(\prod F\nucomp)$.  
Each $\NR_{F}/\nu$ is a finite field of characteristic $p$, and so
$K(\NR_{F}/\nu)\phat$ is an Eilenberg-Mac Lane spectrum.  The Quillen
localization sequence then implies that the map $L_{K(1)}K(\NR_{F})\to
L_{K(1)}K(\RS)$ is a weak equivalence.  Since
\[
\prod R\nucomp\iso (\NR_{F})^{\wedge}_{p}\iso \NR_{F}\otimes \bZ\phat,
\]
Hesselholt-Madsen~\cite[Add.~6.2]{HM2} shows that the map
$TC(\NR_{F})\phat \to TC(\prod R\nucomp)\phat$ is a weak
equivalence.  Hesselholt-Madsen~\cite[Theorem~D]{HM2} shows that the
map $K(\prod R\nucomp)\phat\to TC(\prod R\nucomp)$ is
a connective cover.  Applying the main theorem above and the (affirmed)
Quillen-Lichtenbaum conjecture, we obtain the following corollary.

\begin{maincor}
Let $F$ be a number and let $p\in \bZ$ be an odd prime.
Then there is a canonical map in the stable
category from the homotopy fiber of the cyclotomic trace
$K(\NR_{F})\phat\to TC(\NR_{F})\phat$ to
$\Sigma^{-1}I_{\bZ\phat}(L_{K(1)}K(\NR_{F}))$ that induces an isomorphism
on homotopy groups in dimensions $\geq 2$.
\end{maincor}

Finally, we note that the main theorem implies in the affirmative a
conjecture of Calegari~\cite[1.5]{Calegari} regarding the homotopy
groups of the homotopy fiber on algebraic $K$-theory of the completion
map.  Calegari was interested in the completed cohomology of
$\mathrm{SL}$, and our result yields the construction of a spectrum
with homotopy groups $\tilde{K}_*(\NR_F)$ and continuous spectrum
homology given by $\tilde{H}_*(\SL, \bZp)$~\cite[1.19]{Calegari}.
Moreover, the affirmed conjecture sharpens~\cite[0.2]{Calegari} (an
explicit calculation of completed homology) by making the conclusion
unconditional.

\subsection*{Acknowledgments}
The authors thank Bill Dwyer, Mike Hopkins, Lars Hesselholt, John
Rognes, and Matthias Strauch for helpful conversations, and the
Hausdorff Research Institute for Mathematics for its hospitality while
some of this work was done.  This draft was improved by comments from
Calvin Woo and several anonymous referees.

\section{\texorpdfstring{$K$}K-theoretic local duality and the construction of the map\texorpdfstring{ $\HFC\to \Sigma^{-1}I_{\bZp}\bS$}{}}
\label{sec:map}

The Tate-Poitou duality theorem in global arithmetic derives from a
much easier local duality theorem.  We have a corresponding
$K$-theoretic local duality theorem.  We state and prove the $K$-theoretic local
duality theorem in this section, deducing it from the local duality
theorem in arithmetic.  The argument is parallel to the argument used
in Section~\ref{sec:main} to prove the $K$-theoretic Tate-Poitou
duality theorem and explains the construction of the canonical map
$\HFC\to \Sigma^{-1}I_{\bZp}\bS$, which is characterized by its
relationship to a corresponding map $L_{K(1)}K(\bQp)\to I_{\bZp}\bS$
in the $K$-theoretic local duality of $\bQp$.

In arithmetic, local duality is an isomorphism
\[
H^{i}_{\et}(k;M)\overto{\iso} (H^{2-i}_{\et}(k;M^{*}(1)))^{*}
\]
where $k$ is the field of fractions of a complete discrete valuation
ring whose residue field is finite (e.g., a finite extension of
$\bQp$), $M$ is a finite Galois module, and $(-)^{*}$ denotes the
Pontryagin dual.  The map is induced by the cup product pairing
\[
H^{i}_{\et}(k;M)\otimes H^{2-i}_{\et}(k;M^{*}(1))\to
H^{2}_{\et}(k;M\otimes M^{*}(1))
\to H^{2}_{\et}(k,\bQ/\bZ(1))
\]
and the canonical isomorphism
\begin{equation}\label{eq:Hasse}
H^{2}_{\et}(k,\bQ/\bZ(1))\overto{\iso} H^{2}_{\et}(k,\bG_{m})\overto{\iso} \bQ/\bZ.
\end{equation}
Letting $M=\bZ/p^{n}(j)$ and taking the limit
$n\to \infty$, we get an isomorphism
\[
H^{i}_{\et}(k;\bZp(j))\iso (H^{2-i}_{\et}(k;\bZpi(1-j)))^{*}
\] 
(where the group on the left is Jannsen's continuous \'etale
cohomology, which in the case of a field as above is equivalent to
Galois cohomology). 

For $K$-theory, the $E_{\infty}$ multiplication induces a map
\[
L_{K(1)}K(k) \smaL L_{K(1)}K(k) \smaL M_{\bZpi}\to L_{K(1)}K(k)\smaL M_{\bZpi}
\]
and we have a map 
\begin{equation}\label{eq:canmap}
L_{K(1)}K(k)\smaL M_{\bZpi}\to I_{\bQ/\bZ}\bS
\end{equation}
essentially induced by the Hasse invariant as follows.  Such a map is
uniquely determined by specifying a homomorphism 
\[
\pi_{0}(L_{K(1)}K(k) \smaL M_{\bZpi})\to \bQ/\bZ.
\]
Thomason's descent spectral sequence puts $\pi_{0}(L_{K(1)}K(k) \smaL M_{\bZpi})$
into a short exact sequence
\[
0\to H_{\et}^{2}(k;\bZpi(1))\to \pi_{0}(L_{K(1)}K(k)\smaL M_{\bZpi})\to
H_{\et}^{0}(k;\bZpi)\to 0
\]
which is split by the map $\bZpi=H^{0}_{\et}(k;\bZpi)\to
\pi_{0}(L_{K(1)}K(k) \smaL M_{\bZpi})$ induced by the unit $\bS\to
K(k)$.  This gives us a retraction
\[
\pi_{0}(L_{K(1)}K(k)\smaL M_{\bZpi})\to H_{\et}^{2}(k;\bZpi(1))
\]
that we compose with~\eqref{eq:Hasse} to obtain a map
\begin{equation}\label{eq:pizeromap}
\pi_{0}(L_{K(1)}K(k) \smaL M_{\bZpi})\to \bQ/\bZ.
\end{equation}
The map~\eqref{eq:canmap} represents the map~\eqref{eq:pizeromap}.

\begin{thm}[$K$-Theoretic Local Duality]
\label{thm:localduality}
Let $k$ be the field
of fractions of a complete discrete valuation ring whose residue field
is finite.  The map 
\[
L_{K(1)}K(k) \to I_{\bQ/\bZ}(L_{K(1)}K(k) \smaL M_{\bZpi})\simeq I_{\bZp}(L_{K(1)}K(k))
\]
adjoint to the composite map
\[
L_{K(1)}K(k)\smaL L_{K(1)}K(k)\smaL M_{\bZpi}\to I_{\bQ/\bZ}\bS
\]
described above is a weak equivalence.
\end{thm}

\begin{proof}
Because $L_{K(1)}K(k)$ and $I_{\bZp}(L_{K(1)}K(k))$ are both
$p$-complete, it suffices to check that the map is a weak equivalence
after taking the derived smash product with the mod $p$ Moore spectrum
$M_{p}$, or equivalently taking the homotopy cofiber of multiplication
by $p$.  Using the canonical weak equivalence 
\[
I_{\bQ/\bZ}(L_{K(1)}K(k)\smaL M_{\bZpi})/p \simeq I_{\bQ/\bZ}(L_{K(1)}K(k)/p),
\]
the induced map 
\[
L_{K(1)}K(k)/p\to I_{\bQ/\bZ}(L_{K(1)}K(k)/p)
\]
is adjoint to the composite map
\[
L_{K(1)}K(k)/p\smaL L_{K(1)}K(k)/p \to L_{K(1)}K(k)/p\to
L_{K(1)}K(k)\smaL M_{\bZpi}\to I_{\bQ/\bZ}\bS
\]
induced by the $E_{\infty}$ pairing on $L_{K(1)}K(k)$ and the usual
pairing $M_{p}\smaL M_{p}\to M_{p}$ for the first map and by the
usual inclusion of $M_{p}$ in $M_{\bZpi}$ for the second map.
Thus, it suffices to check that the pairing above
induces a perfect pairing 
\[
\pi_{q}(L_{K(1)}K(k)/p)\otimes \pi_{-q}(L_{K(1)}K(k)/p)\to \bQ/\bZ.
\]
Thomason's descent spectral
sequence~\cite[4.1]{ThomasonEtale} is multiplicative with the
multiplication on the $E_{2}$-term induced by the cup product in
\'etale cohomology.  We therefore have a
perfect pairing on the $E_{2}=E_{\infty}$-term by local duality in
arithmetic, and it follows that we have a perfect pairing on homotopy
groups.
\end{proof}

The proof of the global $K$-theoretic Tate-Poitou duality theorem in
Section~\ref{sec:main} follows the same general outline as the proof
of the local duality theorem above.  However, instead of using the
multiplication on a single $K(1)$-localized $K$-theory spectrum, we
study the pairing of $L_{K(1)}K(\NR_{F}[1/p])$ with $\HFC$, where
$\kappa$ denotes the map $L_{K(1)}K(\RS) \to \prod
L_{K(1)}K(F\nucomp)$ as in the main theorem on
page~\pageref{thm:main}.
In local duality, the map $\pi_{0}(L_{K(1)}K(k)\smaL M_{\bZpi})\to
\bQ/\bZ$ comes from the Hasse invariant isomorphism
$H^{2}_{\et}(k;\bQ/\bZ(1))\iso \bQ/\bZ$, or in terms of $p$-torsion,
the isomorphism 
$H^{2}_{\et}(k;\bZpi(1))\iso \bZpi$.  For global duality, the map
is related to the Albert-Brauer-Hasse-Noether sequence for
$\NR_{F}[1/p]$: the $p$-torsion version of this sequence takes the
form of an exact sequence
\[
0\to H^{2}_{\et}(\NR_{F}[\tfrac1p];\bZpi(1))\to
\prod_{\nu \in S}H^{2}_{\et}(F\nucomp;\bZpi(1))
\to \bZpi\to 0
\]
where $S$ is the set of primes lying above $p$ and $p>2$.  Looking at
the map of short exact sequences from Thomason's descent spectral sequence
\[
\xymatrix@R-1.25pc{%
0\ar[d]&0\ar[d]\\
H^{2}_{\et}(\NR_{F}[\tfrac1p];\bZpi(1))\ar[d]\ar[r]
&\prod H^{2}_{\et}(F\nucomp;\bZpi(1))\ar[d]\\
\pi_{0}(L_{K(1)}K(\NR_{F}[\tfrac1p])\smaL M_{\bZpi})\ar[d]\ar[r]
&\prod \pi_{0}(L_{K(1)}K(F\nucomp)\smaL M_{\bZpi})\ar[d]\\
H^{0}_{\et}(\NR_{F}[\tfrac1p];\bZpi)\ar[d]\ar[r]
&\prod H^{0}_{\et}(F\nucomp;\bZpi)\ar[d]\\
0&0
}
\]
we get an induced map from $\bZpi$ to the cokernel
\[
C=\coker(\pi_{0}(L_{K(1)}K(\NR_{F}[\tfrac1p])\smaL M_{\bZpi})\sto \pi_{0}(\prod
L_{K(1)}K(F\nucomp)\smaL M_{\bZpi})).
\]
This and the long exact sequence on homotopy groups gives us a canonical
map from $\bZpi$ to $\pi_{-1}(\HFC\smaL M_{\bZpi})$.  

\begin{thm}\label{thm:master}
The canonical map $\bZpi\to \pi_{-1}(\HFC\smaL M_{\bZpi})$ is a split
injection and has a unique
retraction that commutes with the $K$-theory transfer associated to
inclusions of number fields.
\end{thm}

\begin{proof}
First we note that the $K$-theory transfer associated to the inclusion
of number fields extends to a well-defined map in the stable category
on $\HFC$:  For $F\subset E$ an inclusion of number fields,
$\NR_{E}[1/p]$ is a finitely generated projective $\NR_{F}[1/p]$-module and we have an associated
$K$-theory transfer map $K(\NR_{E}[1/p])\to K(\NR_{F}[1/p])$ induced
by regarding a finitely generated projective $\NR_{E}[1/p]$-module as
a finitely generated projective $\NR_{F}[1/p]$-module. 
For the $p$-completions
\[
\NR_{E}\otimes \bQp \iso \prod_{\nu\in S_{E}}E\nucomp, \qquad 
\NR_{F}\otimes \bQp \iso \prod_{\nu\in S_{F}}F\nucomp,
\]
we have an associated $K$-theory transfer map and in the standard
models for $K$-theory, the diagram
\[
\xymatrix@-1pc{%
K(\NR_{E}[\tfrac1p])\ar[r]\ar[d]&K(\NR_{E}\otimes \bQp)\ar[d]\\
K(\NR_{F}[\tfrac1p])\ar[r]&K(\NR_{F}\otimes \bQp)
}
\]
commutes up to canonical homotopy since for a finitely generated
projective $\NR_{E}[1/p]$-module $P$, the underlying $(\NR_{F}\otimes
\bQp)$-module of $P\otimes \bQp$ is canonically isomorphic to the
underlying $\NR_{F}[1/p]$-module of $P$ tensored with $\bQp$.  This is
enough structure to specify a canonical map in the stable category on
the homotopy fibers.

Uniqueness is clear because in the case $F=\bQ$, the inclusion of
$\bZpi$ in $\pi_{-1}(\HFC\smaL M_{\bZpi})$ is an isomorphism.  To see
this note that the map from $\bZpi$ to the cokernel $C$ is
an isomorphism  (because the map 
\[
H^{0}(\bZ[1/p];\bZpi)\to
H^{0}(\bQp;\bZpi)
\]
is an isomorphism) and the inclusion of the cokernel
$C$ in $\pi_{-1}(\HFC\smaL M_{\bZpi})$ from the long exact sequence of
homotopy groups is surjective because the map
\[
\pi_{-1}(L_{K(1)}K(\bZ[1/p])\smaL M_{\bZpi})\to
\pi_{-1}(L_{K(1)}K(\bQp)\smaL M_{\bZpi})
\]
is injective
(because the map
$H^{1}(\bZ[1/p];\bZpi)\to H^{1}(\bQp;\bZpi)$ is injective by
abelianized Galois group considerations).

For existence of the splitting, since $\bQ$ is initial among
number fields, we just need to know that for an
inclusion of number fields $F\subset E$, the diagram
\[
\xymatrix@-1pc{%
\bZpi\ar[r]\ar[d]_{\id}&C_{E}\ar[d]\\
\bZpi\ar[r]&C_{F}
}
\]
commutes where $C_{E}$ and $C_{F}$ are the cokernels $C$ associated to
$E$ and $F$ above and the map is induced by the $K$-theory transfer.
Because $H^{2}_{\et}(F\nucomp;\bZpi(1))$ is the $p$-torsion in
$H^{2}_{\et}(F\nucomp;\bG_{m})$, the basic properties of a class
formation (q.v.\ Proposition~1(ii) in~\cite[XI\S2]{Serre-LocalFields})
imply that it is enough to see that the diagram
\[
\xymatrix{%
H^{2}_{\et}(\NR_{E}\otimes \bQp;\bZpi(1))\ar[r]\ar[d]
&\pi_{0}(L_{K(1)}(\NR_{E}\otimes \bQp)\smaL M_{\bZpi})\ar[d]\\
H^{2}_{\et}(\NR_{F}\otimes \bQp;\bZpi(1))\ar[r]
&\pi_{0}(L_{K(1)}(\NR_{F}\otimes \bQp)\smaL M_{\bZpi})
}
\]
commutes where the left vertical map is the transfer in \'etale
cohomology.   
This follows from the well-known result that the
$K$-theory transfer for Galois extensions induces the \'etale
cohomology transfer on the $E_{2}$-page of
Thomason's descent spectral sequence.
\end{proof}

The
composite of the map $\pi_{-1}(\HFC\smaL M_{\bZpi})\to \bZpi$ in the
previous theorem with
the inclusion of $p$-torsion $\bZpi\to \bQ/\bZ$ now specifies a map 
\begin{equation}\label{eq:uadj}
\HFC\smaL M_{\bZpi}\to \Sigma^{-1} I_{\bQ/\bZ}\bS.
\end{equation}
Adjoint to this map is the map
\begin{equation}\label{eq:u}
u_{\NR_{F}}\colon \HFC\to F(M_{\bZpi},\Sigma^{-1}
I_{\bQ/\bZ}\bS)\simeq 
\Sigma^{-1}I_{\bZp}\bS
\end{equation}
in the statement of the $K$-theoretic Tate-Poitou
duality theorem on page~\pageref{thm:main}.  In particular, we have
constructed $u_{\NR_{F}}$ to be
compatible with the corresponding map 
\begin{equation}\label{eq:v}
v_{F\nucomp}\colon L_{K(1)}K(F\nucomp)\to I_{\bZp}\bS
\end{equation}
for local duality adjoint to the map~\eqref{eq:canmap} (for
$k=F\nucomp$).  They are compatible in the sense that $v_{F\nucomp}$
is the composite  
\[
L_{K(1)}K(F\nucomp)\to \Sigma \HFC\overto{\Sigma u_{\NR_{F}}}
\Sigma \Sigma^{-1}I_{\bZp}\bS\iso I_{\bZp}\bS,
\]
where the first map is a component of the map 
\[
\prod_{\nu\in S}L_{K(1)}K(F\nucomp)\to \Sigma \HFC
\]
in the cofiber sequence (associated to the fiber sequence) defining $\HFC$.

\section{\texorpdfstring{$\HFC$ as hypercohomology and $j_{!}$}{Fib(k) as hypercohomology and j!}}\label{sec:dual}

The proof of the $K$-theoretic Tate-Poitou duality
theorem relies on an \'etale hypercohomological interpretation of
$\HFC$.  Arithmetic Tate-Poitou duality arises from a duality pairing
plus a long exact sequence arising from \textit{recollement}.  The
purpose of this section is to give a spectral lifting of this setup.  We
begin with a terse review.

For a fixed number field $F$ and $S$ the set of
primes lying over $p$, let $Y=\spec
\NR_{F}$, let $U$ be the open subscheme $Y\setminus S =\spec (\NR_{F}[1/p])$,
and let $Z$ be the reduced closed 
subscheme $Y\setminus U=\coprod \spec (\NR_{F}/\nu)$.  Writing $i$ for
the inclusion of $Z$ in $Y$ 
and $j$ for the inclusion of $U$ in $Y$, we have various adjoint
functors on sheaves of abelian groups on the \'etale sites:
\begin{equation}\label{eq:reco}
\begin{gathered}
\xymatrix{%
\AbSh{Z_{\et}}
\ar@{<-}[r]<1em>|{i^{*}}
\ar[r]|{i_{*}}
\ar@{<-}[r]<-1em>|{i^{!}}
&\AbSh{Y_{\et}}
\ar@{<-}[r]<1em>|{j_{!}}
\ar[r]|{j^{*}}
\ar@{<-}[r]<-1em>|{j_{*}}
&\AbSh{U_{\et}},
}
\end{gathered}
\end{equation}
where each functor is the left adjoint of the functor below
it.  One consequence of recollement is that for any sheaf or complex
of sheaves $\aF$ on
$Y_{\et}$, the unit of the $i_{*},i^{*}$ adjunction and the counit of
the $j_{!},j^{*}$ adjunction fit into a short exact sequence 
\begin{equation}\label{eq:jses}
0\to j_{!}j^{*}\aF\to \aF\to i_{*}i^{*}\aF\to 0.
\end{equation}
Now we take $\aF$ to be a complex modeling $Rj_{*}(\bZ/p^{n}(t))$, the
total right derived functor of $j_{*}$ applied to the locally constant
sheaf $\bZ/p^{n}(t)$ on $U_{\et}$.  We can
identify the terms in the resulting long exact sequence on
hypercohomology as
\begin{equation}\label{eq:j!}
\cdots \to H^{s}(Y_{\et};j_{!}\bZ/p^{n}(t))\to
H^{s}(U_{\et};\bZ/p^{n}(t))\to
\prod_{\nu\in S} H^{s}_{\et}(F^{h}_{\nu};\bZ/p^{n}(t))\to \cdots 
\end{equation}
(cf.~\cite[II.2.3(a)]{Milne-Duality2007}),
where $F^{h}_{\nu}$ denotes the field of fractions of the
henselization $R^{h}_{\nu}$ of the discrete valuation ring
$(\NR_{F})_{(\nu)}$. $F^{h}_{\nu}$ consists of the elements in the completion
$F\nucomp$ that are algebraic over $F$.  Because the inclusion of
$F^{h}_{\nu}\to F\nucomp$ induces an isomorphism of absolute Galois
groups, it induces an isomorphism
$H^{s}_{\et}(F^{h}_{\nu};\bZ/p^{n}(t))\to
H^{s}_{\et}(F\nucomp;\bZ/p^{n}(t))$.

Tate-Poitou duality is a consequence of the long exact
sequence~\eqref{eq:j!} and the perfect
pairing~\cite[II.3.2--3]{Milne-Duality2007} 
\begin{multline}\label{eq:av}
H^{s}(U_{\et};\bZ/p^{n}(t))
\otimes
H^{3-s}(Y_{\et};j_{!}\bZ/p^{n}(1-t))
\to H^{3}(Y_{\et};j_{!}\bZ/p^{n}(1))\\
\to H^{3}(Y_{\et},j_{!}\bG_{m})\iso \bQ/\bZ.
\end{multline}
Here the
isomorphism $H^{3}(Y_{\et},j_{!}\bG_{m})\iso \bQ/\bZ$ is induced by
the map from 
\[
C'=\coker(H^{2}(U_{\et};\bG_{m})\sto \prod H^{2}_{\et}(F^{h}_{\nu};\bG_{m}))
\]
to $H^{3}(Y_{\et},j_{!}\bG_{m})$ (in the
corresponding long exact sequence for $\aF=Rj_{*}(\bG_{m})$), which is an
isomorphism, together with the canonical isomorphism from $C'$ to 
\[
C=\coker(H^{2}(U_{\et};\bG_{m})\sto \prod H^{2}_{\et}(F\nucomp;\bG_{m}))
\]
and the Albert-Brauer-Hasse-Noether isomorphism from $C$ to
$\bQ/\bZ$.

In light of the above, the first step for $K$-theoretic Tate-Poitou
duality is to identify $\HFC$ in terms of a spectral version of
$j_{!}$ applied to the $K(1)$-local algebraic $K$-theory hypersheaf on
$U_{\et}$.  For a Grothendieck site $\Site$, we write $\HS{\Site}$ for
the $\infty$-category of \term{hypersheaves of spectra} on $\Site$; we
understand this as the full subcategory of the $\infty$-category
$\PS{\Site}$ of presheaves of spectra that satisfy hypercover descent.
This also admits a description in terms of localization:
For a presheaf of spectra $\aF$, let $\tilde\pi_{n}\aF$ denote the
sheafification of the presheaf of abelian groups $\pi_{n}\aF$
(homotopy groups applied objectwise).  Work of
Jardine~\cite{Jardine-SimplicialPresheaves,Jardine-Stable} and
Dugger-Hollander-Isaksen~\cite[1.1]{DHI-Hypercovers} identifies
$\HS{\Site}$ as the localization of $\PS{\Site}$ obtained by formally
inverting the maps that are isomorphisms on $\tilde\pi_{n}$ for all
$n$; cf.~\cite[6.5.3.13]{Lurie-HTT}.  The localization functor
$\PS{\Site}\to \HS{\Site}$ is called \term{hypersheafification}. 

As an example, Thomason~\cite[2.45 or 2.50]{ThomasonEtale} shows that
under the hypotheses that hold there and in particular in our current
setting, $K(1)$-local $K$-theory is a hypersheaf on the small \'etale
site.  We write $\LKF$ for the $K(1)$-localized $K$-theory functor and
$\LK$ for the hypersheaf on $U_{\et}$.  It will also be convenient to
write $\LKn$ for $\LK/p^{n}\simeq \LK\smaL M_{p^{n}}$ and $\LKFn$ for
$\LKF/p^{n}$.   

The recollement above extends to the context of hypersheaves of
spectra; see~\cite[A.8.20,A.8.19]{Lurie-HA}). We have a diagram of
adjoint pairs of functors of hypersheaves of spectra
\[
\xymatrix{%
\HS{Z_{\et}}
\ar@{<-}[r]<1em>|{i^{*}}
\ar[r]|{i_{*}}
\ar@{<-}[r]<-1em>|{i^{!}}
&\HS{Y_{\et}}
\ar@{<-}[r]<1em>|{j_{!}}
\ar[r]|{j^{*}}
\ar@{<-}[r]<-1em>|{j_{*}}
&\HS{U_{\et}}
}
\]
mostly analogous to the diagram of adjoint pairs of functors of
sheaves of abelian groups pictured in~\eqref{eq:reco}, or more
precisely analogous to the derived category extension.
(The functors $i^{!}$ and $j_{*}$ are the analogs of the right
derived functors $Ri^{!}$ and $Rj_{*}$ on the derived categories of
sheaves of abelian groups.)  The analogue of~\eqref{eq:jses} also
holds:
\[
j_{!}j^{*}\aF \simeq \Fib(\aF\sto i_{*}i^{*}\aF)
\]
(see the proof of (b) in \cite[A.8.20]{Lurie-HA}, where the equivalent
adjoint formula is proved).  In particular, taking $\aF$ to be
$j_{*}\LK$, we have
\[
j_{!}\LK \simeq \Fib(j_{*}\LK\sto i_{*}i^{*}j_{*}\LK).
\]
The following theorem relates $j_{!}\LK$ to $\HFC$, the
homotopy fiber of the completion map.  

\begin{thm}\label{thm:bang}
The spectrum $j_{!}\LK(Y_{\et})$ of global sections of $j_{!}\LK$
is $p$-equivalent to $\HFC$.
\end{thm}

\begin{proof}
We construct a commutative diagram 
\[
\xymatrix{%
j_{*}\LK(Y_{\et})\ar[r]\ar[d]_{\simeq}
&i_{*}i^{*}j_{*}\LK(Y_{\et})\ar[d]\\
\LKF(U_{\et})\ar[r]
&\displaystyle\prod_{\nu\in S}\LKF(F^{h}_{\nu})\\
}
\]
with the left vertical map the tautological equivalence and the right
vertical map a $p$-equivalence constructed below.  This will then complete the 
argument since the completion map $\kappa$ factors as the bottom
horizontal map followed by the map
\[
\prod_{\nu\in S}\LKF(F^{h}_{\nu})\to \prod_{\nu\in S}\LKF(F\nucomp)
\]
induced by the inclusions $F^{h}_{\nu}\to F\nucomp$, which is a
$p$-equivalence by Thomason's theorem~\cite[4.1]{ThomasonEtale}
proving the $K(1)$-local Quillen-Lichtenbaum conjecture.

Construction of the righthand map essentially amounts to understanding the
hypersheaf $i_{*}i^{*}j_{*}\LK$.  Write $i^{*}_{\pre}$ for the
inverse image functor $\PS{Y_{\et}}\to \PS{Z_{\et}}$. Then for $\nu\in
S$, let $R^{h}_{\nu}$ denote the henselization of $(\NR_{F})_{(\nu)}$,
and for any finite separable extension $k$ of $\NR_{F}/\nu$, let
$R^{h}_{\nu}(k)$ denote the corresponding \'etale 
$R^{h}_{\nu}$-algebra (under the usual equivalence of
categories~\cite[I.4.4]{Milne-EtaleCohomology}).  Then we have 
\[
i^{*}_{\pre}j_{*}\LKn(\spec k) \simeq \LKFn((\spec
R^{h}_{\nu}(k))\times_{Y}U) \simeq \LKFn(F^{h}_{\nu}(k))
\]
(cf.~\cite[1.44]{ThomasonEtale})
where $F^{h}_{\nu}(k)=R^{h}_{\nu}(k)[1/p]$ is the quotient field. It
follows that $i^{*}_{\pre}j_{*}\LKn$ satisfies hypercover descent and
so computes $i^{*}j_{*}\LKn$. In particular
$i_{*}i^{*}j_{*}\LKn(Y_{\et})\simeq \prod \LKn(F^{h}_{\nu})$. Similarly,
$(i^{*}j_{*}\LK)\phat\simeq (i^{*}_{\pre}j_{*}\LK)\phat$, and 
this induces the $p$-equivalence in the diagram.  Since the map 
\[
\LKF(\NR_{F}[1/p])\simeq j_{*}\LK(Y_{\et})\to i_{*}i^{*}_{\pre}j_{*}\LK(Y_{\et})\to \LKF(F^{h}_{\nu})
\]
is induced by the map $\NR_{F}[1/p]\to F^{h}_{\nu}$, the diagram
commutes. 
\end{proof}

In order to apply the previous theorem, we also need to know how the
pairing $\LKF(\NR_{F}[1/p])\sma \HFC \to \HFC$ relates to the
interpretation of $\HFC$ as $j_{!}\LK(Y_{\et})\phat$.  Recent work of
Clausen-Mathew~\cite[2.17]{ClausenMathew} proves that the hypersheafication functor
on any Grothendieck site is lax symmetric monoidal; it follows that
the $\infty$-categories of hypersheaves of spectra on Grothendieck
sites are symmetric monoidal $\infty$-categories.  The direct image
functor is then a symmetric monoidal functor; in particular, the
natural multiplication on $K$-theory induces pairings of
hypersheaves (natural in $n$)
\begin{gather*}
j_{*}\LK \sma j_{*}\LK \to j_{*}\LK \\
j_{*}\LKn\sma j_{*}\LKn\to j_{*}\LKn
\end{gather*}
compatible with the usual pairings
\begin{gather*}
\LKF(\NR_{F}[1/p])\sma \LKF(\NR_{F}[1/p])\to \LKF(\NR_{F}[1/p])\\
\LKFn(\NR_{F}[1/p])\sma \LKFn(\NR_{F}[1/p])\to \LKFn(\NR_{F}[1/p])
\end{gather*}
when passing to global sections.  The usual (equivalence on stalks)
argument shows that for any hypersheaves of spectra $\aF$ and $\aG$,
\[
j_{*}\aF \sma j_{!}\aG \simeq j_{!}(\aF \sma \aG),
\]
and in our context, this and the pairings above give pairings (natural in $n$)
\begin{equation}\label{eq:bangpair}
\begin{gathered}
j_{*}\LK \sma j_{!}\LK \to j_{!}(\LK \sma \LK) \to j_{!}\LK\\
j_{*}\LKn \sma j_{!}\LKn \to j_{!}(\LKn \sma \LKn) \to j_{!}\LKn.
\end{gathered}
\end{equation}

\begin{thm}\label{thm:compair}
Under the equivalence of Theorem~\ref{thm:bang}, the pairing
$\LKF(\NR_{F}[1/p])\sma \HFC\to \HFC$ is the induced pairing on global
sections from the pairing of hypersheaves of~\eqref{eq:bangpair}.
\end{thm}

\begin{proof}
The equivalence $j_{*}\aF \sma j_{!}\aG \simeq j_{!}(\aF \sma \aG)$ is
the inverse of the map adjoint to the equivalence
\begin{equation}\label{eq:jbangsmadef}
j^{*}(j_{*}\aF \sma j_{!}\aG)\simeq j^{*}j_{*}\aF \sma j^{*}j_{!}\aG
\simeq \aF \sma \aG;
\end{equation}
we need to see that it is the map
\begin{multline}\label{eq:jbangsmause}
j_{*}\aF \sma \Fib(j_{*}\aG\sto i_{*}i^{*}j_{*}\aG)
\simeq
\Fib\bigl((j_{*}\aF \sma j_{*}\aG) \sto
   (j_{*}\aF \sma i_{*}i^{*}j_{*}\aG)\bigr)
\to \\
\Fib\bigl((j_{*}\aF \sma j_{*}\aG) \sto
   (i_{*}i^{*}j_{*}\aF \sma i_{*}i^{*}j_{*}\aG)\bigr)
\simeq
\Fib\bigl(j_{*}(\aF\sma \aG)\sto i_{*}i^{*}j_{*}(\aF\sma\aG)\bigr)
\end{multline}
under the usual identification of $j_{!}$ with the fiber.  Now it is
easy to see that~\eqref{eq:jbangsmause} is inverse to the adjoint
of~\eqref{eq:jbangsmadef} by applying $j^{*}$.
\end{proof}

\section{Proof of the \texorpdfstring{$K$}K-theoretic Tate-Poitou duality theorem}\label{sec:main}

In this section we prove the $K$-theoretic Tate-Poitou duality
theorem.  We deduce the result from the classical Tate-Poitou duality
theorem; more precisely, we use the formulation in terms of
Artin-Verdier duality~\eqref{eq:av}.  We argue in terms of a pairing
of \'etale descent spectral sequences.  

\begin{thm}\label{thm:multss}
The descent spectral sequences
\begin{align*}
E_{2}^{s,t}(U_{\et};\LKn)&=H^{s}(U_{\et};\bZ/p^{n}(t/2))
&&\Longrightarrow\quad \pi_{-s+t}\LKn(U_{\et})&\text{and}\\
E_{2}^{s,t}(Y_{\et};j_{!}\LKn)&=H^{s}(Y_{\et};j_{!}\bZ/p^{n}(t/2))
&&\Longrightarrow\quad \pi_{-s+t}j_{!}\LKn(Y_{\et})
\end{align*}
admit a pairing of the form
\[
E_{r}^{s,t}\LKn(U_{\et})\otimes E_{r}^{s',t'}j_{!}\LKn(Y_{\et})
\to E_{r}^{s+s',t+t'}j_{!}\LKn(Y_{\et})
\]
which converges to the pairing
\[
\pi_{-s+t}\LKn(U_{\et})\otimes \pi_{-s'+t'}j_{!}\LKn(Y_{\et})
\to \pi_{-(s+s')+(t+t')}j_{!}\LKn(Y_{\et})
\]
induced from the weak equivalences 
\[
\LKn(U_{\et})\simeq\LKFn(\spec \NR_{F}[\tfrac1p]),
\qquad\textrm{and}\qquad 
j_{!}\LKn(Y_{\et})\simeq \HFC^{/p^{n}}
\]
and the pairing
$\LKFn(\spec \NR_{F}[\tfrac1p]) \smaL \HFC^{/p^{n}} \to \HFC^{/p^{n}}$.
\end{thm}

We prove this theorem in the next section.
In order to apply it, we need to related the pairing on the
$E^{2}$-term with the pairing in Artin-Verdier duality.  We also prove
the following theorem in the next section.

\begin{thm}\label{thm:identmult}
Under the canonical isomorphism 
\begin{multline*}
H^{*}(U_{\et};\bZ/p^{n}(t/2))\iso
\Ext^{*}_{U_{\et}}(\bZ/p^{n}(t'/2),\bZ/p^{n}(t/2+t'/2))\\
\iso
\Ext^{*}_{Y_{\et}}(j_{!}\bZ/p^{n}(t'/2),j_{!}\bZ/p^{n}(t/2+t'/2)),
\end{multline*}
the multiplication on the $E_{2}$-term in
Theorem~\ref{thm:multss}
\[
H^{*}(U_{\et};\bZ/p^{n}(t/2))
\otimes 
H^{*}(Y_{\et};j_{!}\bZ/p^{n}(t'/2))
\to 
H^{*}(Y_{\et};j_{!}\bZ/p^{n}(t/2+t'/2))
\]
coincides with the Yoneda pairing
\begin{multline*}
\Ext^{*}_{Y_{\et}}(j_{!}\bZ/p^{n}(t'/2),j_{!}\bZ/p^{n}(t/2+t'/2))
\otimes 
H^{*}(Y_{\et};j_{!}\bZ/p^{n}(t'/2))
\\
\to H^{*}(Y_{\et};j_{!}\bZ/p^{n}(t/2+t'/2))
\end{multline*}
\end{thm}

The previous two theorems give all the ingredients we need to prove
the $K$-theoretic Tate-Poitou duality theorem.

\begin{proof}[Proof of the $K$-theoretic Tate-Poitou duality theorem]
The pairing
\[
\LKF(\spec \NR_{F}[\tfrac1p])\smaL \HFC\to \HFC
\]
and the map 
\[
u_{\NR_{F}}\colon \HFC\to \Sigma^{-1}I_{\bZp}\bS
\]
of~\eqref{eq:u} give a pairing
\[
\LKF(\spec \NR_{F}[\tfrac1p])\smaL \HFC\to \Sigma^{-1}I_{\bZp}\bS,
\]
which induces a map
\begin{equation}\label{eq:dualitymap}
\HFC\to \Sigma^{-1}I_{\bZp}(\LKF(\spec \NR_{F}[\tfrac1p]))\simeq
\Sigma^{-1}I_{\bQ/\bZ}(\LKF(\spec \NR_{F}[\tfrac1p])\smaL M_{\bZpi}).
\end{equation}
We need to see that it is a weak equivalence.  
Since both sides are
$p$-complete, it suffices to check that~\eqref{eq:dualitymap} becomes
a weak equivalence 
after smashing with the mod $p$ Moore spectrum $M_{p}$ on both sides.
Then we are looking at the map
\begin{equation}\label{eq:pdualitymap}
\HFC/p\to \Sigma^{-1}I_{\bQ/\bZ}(\LKF(\spec \NR_{F}[\tfrac1p])\smaL M_{\bZpi})/p
\simeq \Sigma^{-1}I_{\bQ/\bZ}(\LKF(\spec \NR_{F}[\tfrac1p])/p),
\end{equation}
which is adjoint to a map
\begin{equation}\label{eq:pairmodp}
\LKF(\spec \NR_{F}[\tfrac1p])/p\smaL \HFC/p\to \Sigma^{-1}I_{\bQ/\bZ}\bS.
\end{equation}
Naturality and the fact that for an odd prime the map
\[
M_{p}\smaL M_{p}\simeq F(M_{p},M_{\bZpi})\smaL M_{p}\to
M_{\bZpi}
\]
induced by evaluation is the same as the composite of the
multiplication on $M_{p}$ and the inclusion of $M_{p}$ in
$M_{\bZpi}$ imply that the map~\eqref{eq:pairmodp} 
is the composite of the multiplication
\[
\LKF(\spec \NR_{F}[\tfrac1p])/p \smaL \HFC/p\to \HFC/p
\]
and the map 
\[
\HFC/p\to \HFC\smaL M_{\bZpi}\to \Sigma^{-1}I_{\bQ/\bZ}\bS
\]
(where $\HFC\smaL M_{\bZpi}\to \Sigma^{-1}I_{\bQ/\bZ}\bS$ is the
map~\eqref{eq:uadj} adjoint to $u_{\NR_{F}}$).
Because \eqref{eq:av} is a perfect pairing, Theorem~\ref{thm:multss}
and Theorem~\ref{thm:identmult} imply that \eqref{eq:pairmodp}
induces a perfect pairing on homotopy groups
\[
\pi_{q}(\LKF(\spec \NR_{F}[\tfrac1p])/p)\otimes \pi_{-1-q}(\HFC/p)\to \bQ/\bZ.
\]
This implies that \eqref{eq:pdualitymap} is a weak equivalence, and we
conclude that \eqref{eq:dualitymap} is a weak equivalence.  
\end{proof}

\section{Construction and analysis of the spectral sequence}\label{sec:ss}

This section proves Theorems~\ref{thm:multss}
and~\ref{thm:identmult}.  
For the proof of~\ref{thm:multss}, we use a modern take on the descent
spectral sequence of Jardine~\cite[\S6.1]{Jardine-GeneralisedEtale}
based on Postnikov towers rather than the original approach of
Thomason based on the Godement construction; these are well-known to
be isomorphic from $E^{2}$ onwards.  As explained by
Dugger~\cite[\S4]{Dugger-multii}, using Whitehead towers in place of
Postnikov towers leads to the same spectral sequence but with better
multiplicative properties, and we take this approach.

Work of Hedenlund-Krause-Nikolaus~\cite{HKN-mult} expands on Dugger's
observations on pairing of spectral sequences and puts them in modern
language.  Although the category of bigraded spectral sequences of
abelian groups does not form a symmetric monoidal category, it does
form an $\infty$-operad under the usual notion of multilinear pairing.
The main theorem of \cite{HKN-mult} is that the usual construction of
a spectral sequence from a tower of spectra assembles to a map of
$\infty$-operads.  Here the $\infty$-category of towers is the
$\infty$-category of spectral presheaves on the ordered set $\bZ$ (for
the increasing order) and the $\infty$-operad structure is the Day
convolution symmetric monoidal structure for
addition~\cite[4.8.1.13]{Lurie-HA}.  A pairing of towers $A\supdot\sma
B\supdot\to C\supdot$ in this structure amounts to (homotopy coherent)
maps
\[
A^{t}\sma B^{t'}\to C^{t+t'}.
\]
A pairing of towers gives a pairing of spectral sequences in the
classical sense.

The Whitehead tower is the name for the functor obtained by assembling
into a tower the truncations $\tau^{\geq n}$ in a $t$-structure. We
will write this functor as $W\supdot$, i.e., 
$W^{n}:=\tau^{\geq n}$. In the present context, $W\supdot$ is a
functor from the $\infty$-category of spectra (with the standard
$t$-structure) to the $\infty$-category
of towers of spectra.  Because the smash product of spectra adds
connectivities, this can be enhanced to a lax symmetric monoidal
functor of symmetric monoidal $\infty$-categories.  Because 
hypersheafication from presheaves of spectra on $U_{\et}$ to \'etale
hypersheaves of spectra on $U_{\et}$ is a symmetric monoidal
functor~\cite[2.17]{ClausenMathew}, the Whitehead tower functor followed by global
sections gives a symmetric monoidal functor from the $\infty$-category
of \'etale hypersheaves on $Y_{\et}$ to the $\infty$-category
of towers of spectra.

The descent spectral sequences of Theorem~\ref{thm:multss} apply the
above composite map of $\infty$-operads from \'etale hypersheaves of
spectra to spectral sequences.  Concretely, the first spectral
sequence in~\ref{thm:multss} is the homotopy group spectral sequence
of the tower of spectra
\[
\cdots \to (W^{t+1}\LKn)(U_{\et})\to (W^{t}\LKn)(U_{\et})\to \cdots .
\]
As an abbreviation in the work below, we write 
\[
C^{t}\LKn=(W^{t}\LKn,W^{t+1}\LKn)
\]
for the pair $W^{t+1}\LKn\to W^{t}\LKn$.  In particular,
\[
\pi_{*}(C^{t}\LKn(U_{\et}))\iso \pi_{*}\Cof\bigl((W^{t+1}\LKn)(U_{\et})\sto(W^{t}\LKn)(U_{\et})\bigr).
\]
Since
$\tilde\pi_{t}\LKn\iso \bZ/p^{n}(t/2)$, we get a canonical isomorphism 
\[
\pi_{-s+t}(C^{t}\LKn(U_{\et}))\iso H^{s}(U_{\et};\bZ/p^{n}(t/2)).
\]
The spectral sequence then has $E_{2}$-term (with the standard
Whitehead/Postnikov Atiyah-Hirzebruch renumbering) 
\[
E_{2}^{s,t}:=\pi_{-s+t}(C^{t}\LKn(U_{\et})) \iso H^{s}(U_{\et};\bZ/p^{n}(t/2))
\]
and abuts to the colimit
\[
\colim \pi_{-s+t} \bigl((W^{\bullet}\LKn)(U_{\et})\bigr) \iso
\pi_{-s+t}\bigl(\LKn(U_{\et})\bigr). 
\]
Because $\holim (W^{\bullet}\LKn)(U_{\et})\simeq *$, the spectral
sequence converges conditionally~\cite[5.10]{Boardman-SpectralSequences}.
Because $H^{s}(U_{\et};\bZ/p^{n}(t/2))$ is only non-zero in a finite
range, the spectral sequence converges
strongly~\cite[6.1]{Boardman-SpectralSequences} to the abutment
$\pi_{-s+t}\LKn(U_{\et})$. 

For the spectral sequence on $\HFC$, we use the tower
\[
\dotsb \to (j_{!}W^{t+1}\LKn)(Y_{\et})\to (j_{!}W^{t}\LKn)(Y_{\et})\to \dotsb .
\]
We abbreviate
\[
\aX^{t}:=j_{!}W^{t}\LKn \simeq \Fib(j_{*}W^{t}\LKn\sto i_{*}i^{*}j_{*}W^{t}\LKn)
\]
and write
\[
C^{t}\aX=(\aX^{t},\aX^{t+1})
\]
for the pair.  We then have
\[
\pi_{-s+t}C^{t}\aX(Y_{\et})\iso H^{s}(Y_{\et};j_{!}\tilde\pi_{-t}\LKn)
\]
and the tower of spectra $\aX^{\bullet}(Y_{\et})$ gives a spectral 
sequence with 
\[
E_{2}^{s,t}:=\pi_{-s+t}(\aX^{t}(Y_{\et}),\aX^{t+1}(Y_{\et})) 
\iso H^{s}(Y_{\et};j_{!}\bZ/p^{n}(t/2))
\]
that abuts to the colimit
\[
\colim \pi_{-s+t} \aX^{\bullet}(Y_{\et})\iso \pi_{-s+t}\bigl(j_{!}\LKn(Y_{\et})\bigr).
\]
Again because the homotopy limit of $\aX^{\bullet}$ is trivial and
$H^{s}(Y_{\et};j_{!}\bZ/p^{n}(t/2))$ is only non-zero in a finite 
range, the spectral sequence converges strongly, and 
$\pi_{-s+t}\colim \aX^{\bullet}(Y_{\et})$ is isomorphic
to $\pi_{-s+t}\HFC$ by the comparison map.  
The pairing property of
$W^{\bullet}$ induces a pairing
\begin{multline*}
W^{t} \LKn(U_{\et}) \sma \aX^{t'}(Y_{\et})
=j_{*} W^{t} \LKn(Y_{\et}) \sma j_{!}W^{t'}\LKn(Y_{\et})\\
\to j_{!}W^{t+t'} (\LKn(Y_{\et}) \sma \LKn(Y_{\et}))
\to j_{!}W^{t+t'}\LKn(Y_{\et})=
\aX^{t+t'}(Y_{\et}),
\end{multline*}
inducing a pairing of spectral sequences.  Theorem~\ref{thm:compair}
identifies this with our standard model
\[
\LKFn(\spec \NR_{F}[\tfrac1p])\sma \HFC^{/p^{n}}\to \HFC^{/p^{n}}
\]
for the pairing of the $K(1)$-local mod $p^{n}$ algebraic $K$-theory
and the fiber.  

This completes the proof of Theorem~\ref{thm:multss}.
Next we need to identify the multiplication on the $E_{2}$-term, which
takes the form
\[
H^{s}(U_{\et};\bZ/p^{n}(t/2))
\otimes
H^{s'}(Y_{\et};j_{!}\bZ/p^{n}(t'/2))
\to
H^{s+s'}(Y_{\et};j_{!}\bZ/p^{n}(t/2+t'/2)).
\]
In the notation above, the
multiplication is induced by the map of pairs
\[
j_{*}C^{t}\LKn(Y_{\et})\sma C^{t'}\aX(Y_{\et})
\to C^{t+t'}\aX(Y_{\et}).
\]
By construction the homotopy cofiber of the pair 
$j_{*}C^{t}\LKn$ is a model for the \'etale hypersheaf 
$\Sigma^{t}j_{*}H\bZ/p^{n}(t/2)$ on $Y_{\et}$ and the homotopy cofiber of the
pair 
$C^{t'}\aX$ is a model
for the \'etale hypersheaf $\Sigma^{t'}j_{!}H\bZ/p^{n}(t'/2)$ on $Y_{\et}$.
Thus, we can identify the induced map on homotopy groups of global
sections as the composite of the cup
product
\begin{multline*}
H^{*}(U_{\et};\bZ/p^{n}(t'/2))
\otimes 
H^{*}(Y_{\et};j_{!}\bZ/p^{n}(t/2))
\\
\iso
H^{*}(Y_{\et};j_{*}H\bZ/p^{n}(t'/2))
\otimes 
H^{*}(Y_{\et};j_{!}H\bZ/p^{n}(t/2))
\\
\to
H^{*}(Y_{\et};
j_{*}H\bZ/p^{n}(t'/2)
\smaL
j_{!}H\bZ/p^{n}(t/2))
\end{multline*}
and the map of \'etale hypersheaves 
\[
j_{*}H\bZ/p^{n}(t'/2)
\smaL
j_{!}H\bZ/p^{n}(t/2)
\to j_{!}H\bZ/p^{n}(t/2+t'/2)
\simeq Hj_{!}\bZ/p^{n}(t/2+t'/2)
\]
induced by the pairing.  Using the equivalence of \'etale
hypersheaves of spectra
\[
j_{*}H\bZ/p^{n}(t/2)
\sma
j_{!}H\bZ/p^{n}(t'/2)
\simeq
j_{!}(H\bZ/p^{n}(t/2)\sma
H\bZ/p^{n}(t'/2)),
\]
since the target 
\[
j_{!}H\bZ/p^{n}(t/2+t'/2)
\simeq Hj_{!}\bZ/p^{n}(t/2+t'/2)
\]
is an Eilenberg-Mac Lane presheaf, the map is determined by the
factorization through the coconnective cover
\[
j_{!}(H\bZ/p^{n}(t/2)\sma H\bZ/p^{n}(t'/2))(-\infty,0]
\simeq j_{!}(H\bZ/p^{n}(t/2+t'/2)).
\]
By looking at stalks, we see that the self-map of
$j_{!}(H\bZ/p^{n}(t/2+t'/2))$ is the identity.  As a consequence, it
follows that the map
\[
H^{*}(U_{\et};\bZ/p^{n}(t/2))
\otimes 
H^{*}(Y_{\et};j_{!}\bZ/p^{n}(t'/2))
\to 
H^{*}(Y_{\et};j_{!}\bZ/p^{n}(t/2+t'/2))
\]
on the $E_{2}$-term in Theorem~\ref{thm:multss} factors through the
corresponding cup product map in the derived category of sheaves of
abelian groups on $Y_{\et}$,
\begin{multline*}
H^{*}(U_{\et};\bZ/p^{n}(t/2))
\otimes 
H^{*}(Y_{\et};j_{!}\bZ/p^{n}(t'/2))
\\
=
\bH^{*}(Y_{\et};Rj_{*}\bZ/p^{n}(t/2))
\otimes 
H^{*}(Y_{\et};j_{!}\bZ/p^{n}(t'/2))
\\
\to
H^{*}(%
\bH_{\Ab}(Y_{\et};Rj_{*}\bZ/p^{n}(t/2))
\otimes^{L}
\bH_{\Ab}(Y_{\et};j_{!}\bZ/p^{n}(t'/2)))
\\
\to 
\bH^{*}(Y_{\et};Rj_{*}\bZ/p^{n}(t/2)\otimes^{L}j_{!}\bZ/p^{n}(t'/2))
\\
\iso
\bH^{*}(Y_{\et};j_{!}(\bZ/p^{n}(t/2)\otimes^{L}\bZ/p^{n}(t'/2)))\\
\to
H^{*}(Y_{\et},j_{!}(\bZ/p^{n}(t/2)\otimes\bZ/p^{n}(t'/2)))\\
\iso
H^{*}(Y_{\et},j_{!}\bZ/p^{n}(t/2+t'/2)).
\end{multline*}
Here for ease of comparison to algebraic conventions, we have switched
to derived category notation and written $\bH_{\Ab}(Y_{\et};-)$ for
the hypercohomology object of a sheaf of abelian groups (i.e., the
sections of the hypersheaf, viewed as an object of the derived
category of abelian groups) and $\bH^{*}(Y_{\et};-)$ for its
hypercohomology groups $H^{*}(\bH_{\Ab}(Y_{\et};-))$.

This identifies the multiplication on the $E_{2}$
term in terms of the cup product, and Theorem~\ref{thm:identmult} now
follows from the 
basic relationship between the cup product and the Yoneda product in
the derived category of sheaves of abelian groups on $Y_{\et}$,
cf.~\cite[\S5.1]{Milne-EtaleCohomology}:
For sheaves of abelian groups $\aF$ and $\aG$, the following diagram
in the derived category commutes
\[
\xymatrix{%
\bH_{\Ab}(Y_{\et};R\SHom(\aF,\aG))
\otimes^{L}
\bH_{\Ab}(Y_{\et};\aF)
\ar[r]\ar[d]_{\simeq}
&\bH_{\Ab}(Y_{\et};R\SHom(\aF,\aG)\otimes^{L}\aF)\ar[d]\\
R\Hom(\aF,\aG)\otimes^{L}\bH_{\Ab}(Y_{\et};\aF)\ar[r]
&\bH_{\Ab}(Y_{\et};\aG)
}
\]
where the top arrow is the cup product and the bottom arrow and
righthand arrows are the appropriate evaluation maps.
This completes the proof of Theorem~\ref{thm:identmult}.



\bibliographystyle{plain}
\bibliography{bluman}

\end{document}